\documentclass{article}

\usepackage{PRIMEarxiv}

\usepackage[utf8]{inputenc} % allow utf-8 input
\usepackage[T1]{fontenc}    % use 8-bit T1 fonts
\usepackage{hyperref}       % hyperlinks
\usepackage{url}            % simple URL typesetting
\usepackage{booktabs}       % professional-quality tables
\usepackage{amsfonts}       % blackboard math symbols
\usepackage{nicefrac}       % compact symbols for 1/2, etc.
\usepackage{microtype}      % microtypography
\usepackage{lipsum}
\usepackage{fancyhdr}       % header
\usepackage{graphicx}       % graphics
\graphicspath{{media/}}     % organize your images and other figures under media/ \usepackage{times}
\usepackage{grffile}
\usepackage{amsmath,amssymb,amsbsy}
\usepackage{hyperref}
\usepackage{color}
\usepackage{comment}
\usepackage{lipsum}% http://ctan.org/pkg/lipsum
\usepackage[normalem]{ulem}
\usepackage{enumerate}
\usepackage{epstopdf}
\usepackage{bbm}
\usepackage{tikz}
\usepackage[prependcaption,colorinlistoftodos]{todonotes}
\usepackage{multicol}
\usepackage{algorithmic}
\usepackage[]{algorithm2e}
\usepackage{amsthm}
\usepackage[caption=false,font=footnotesize]{subfig}

\newcommand{\x}{\mathbf{x}}
\newcommand{\y}{\mathbf{y}}

\newcommand{\R}{\mathbb{R}}

\newcommand{\A}{\mathcal{A}}

\newtheorem{definition}{Definition}
\newtheorem{problem}{Problem}
\newtheorem{remark}{Remark}
\newtheorem{assumption}{Assumption}

\newcommand{\map}[3]{#1: #2 \rightarrow #3}
\newcommand{\B}{\mathcal{B}}

\newcommand{\D}{\mathcal{D}}
\newcommand{\T}{\mathcal{T}}

\newtheorem{theorem}{Theorem}[section]
\newtheorem{proposition}[theorem]{Proposition}
\newtheorem{lemma}[theorem]{Lemma}

\usepackage{comment}
\title{Denoising Diffusion-Based Control of Nonlinear Systems
}

\author{
  Karthik Elamvazhuthi, Darshan Gadginmath, Fabio Pasqualetti\\
  Department of Mechanical Engineering \\
  University of California, Riverside \\
  \texttt{\{karthike@ucr.edu, dgadg001@ucr.edu,fabiopas\}@engr.ucr.edu} \\
}

\begin{document}
\maketitle

\begin{abstract}
We propose a novel approach based on Denoising Diffusion Probabilistic Models (DDPMs) to  control nonlinear dynamical systems. DDPMs are the state-of-art of generative models that have achieved success in a wide variety of sampling tasks. In our framework, we pose the feedback control problem as a generative task of drawing samples from a target set under control system constraints. The forward process of DDPMs constructs trajectories originating from a target set by adding noise. We learn to control a dynamical system in reverse such that the terminal state belongs to the target set. For control-affine systems without drift, we prove that the control system can exactly track the trajectory of the forward process in reverse, whenever the the Lie bracket based condition for controllability holds. We numerically study our approach on various nonlinear systems and verify our theoretical results. We also conduct numerical experiments for cases beyond our theoretical results on a physics-engine.
\end{abstract}

% keywords can be removed
\keywords{Diffusion Models  \and Nonlinear Control Systems \and  Generative Modeling \and Geometric Control}

\section{Introduction} \label{sec:intro}

Feedback control plays a pivotal role in modern engineering and technology. The essence of feedback control lies in using the value of the state of the system to compute an input that steers a dynamical system to a desired target state. This process is crucial in a multitude of applications. For instance, in robotics, feedback control is fundamental for executing tasks such as pick-and-place operations, where precise movement and placement of objects are required, and in coverage control tasks, which involve area scanning and monitoring. Feedback control is also integral for power network management, where it helps maintain stability and efficiency in the density and generation of power. In each of these scenarios, feedback control is not just about reaching a target state, but doing so in an efficient, robust, and safe manner, which is essential for the functioning of complex modern systems. In this paper, we explore feedback control design from the perspective of generative modeling.

Generative modeling, an important tool in machine learning, addresses the challenge of drawing new samples from an unknown data density when provided with a set of data samples. Among the cutting-edge techniques in this domain, Denoising Diffusion Probabilistic Models (DDPMs) \cite{ho2020denoising,yang2022diffusion} have been successfully applied in diverse fields such as medical imaging \cite{chung2022score}, path planning \cite{yang2022diffusion}, and shape generation \cite{zhou20213d}. DDPMs comprise two essential components: a \emph{forward process}, responsible for mapping the data density to a desired noise density, and a \emph{reverse process}, which accomplishes the inverse transformation by approximately retracing the trajectory of the probability density established by the forward process. This bidirectional approach to sampling distinguishes DDPMs from traditional methods, where only a forward process is employed to map the noise density to the data density. Classical techniques like normalizing flows \cite{kobyzev2020normalizing},  optimal transport-based methods \cite{finlay2020train}, Metropolis-Hastings algorithms \cite{turner2019metropolis} and Langevin sampling \cite{bussi2007accurate} share this unidirectional characteristic. This difference is one of the keys to the broad success of DDPMs as generative models.

For control systems, DDPMs offer a novel and largely unexplored avenue for controller design by reframing control design as a generative modeling task. This approach facilitates the development of a unique method to guide a controlled dynamical system towards a desired probability density. A key benefit of this technique is its ability to concurrently construct feedback controllers and plan trajectories for the system within potentially non-convex environments. Traditional approaches to trajectory planning in controlled dynamical systems usually adopt a dual-layer strategy: one layer for path planning within the spatial domain, and another for designing a local stabilizer that maintains the system's stability around the planned path. However, conventional methods like optimal control or reinforcement learning for direct controller design often face challenges in nonlinear systems, high-dimensional state spaces, and non-convex environments. These challenges echo the performance disparities observed between classical sampling-based methods and the capabilities of DDPMs. This comparison underscores the motivation for our proposed methodology, which leverages DDPMs to simultaneously tackle control design and trajectory planning. To illustrate, we apply this method in navigating a robot through a non-convex environment as depicted in Figure~\ref{fig:unicycle-obst}. 

\begin{figure}
\centering
\begin{multicols}{2}
\begin{tikzpicture}
  \node (img1)  {\includegraphics[width=0.3000\textwidth]{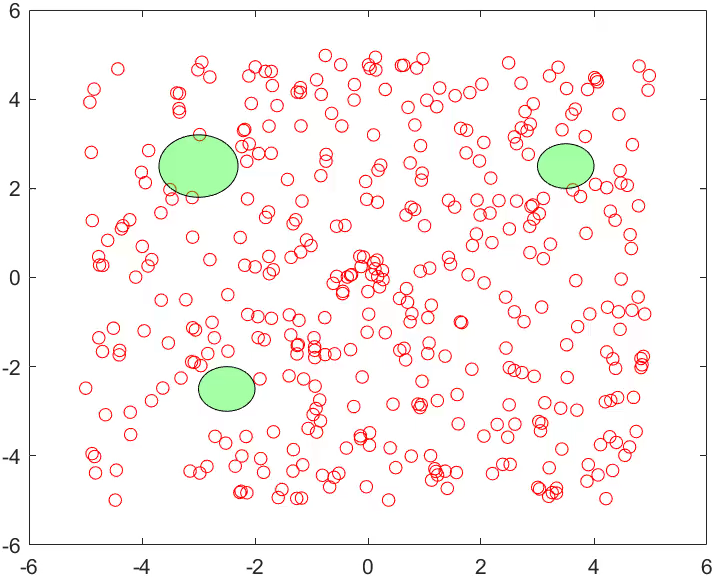}};
  \node[above of= img1, node distance=-0.5cm, yshift=-2.1cm,font=\color{black}]  {Initial density};  
  % \node[below of= img1, node distance=0cm, yshift=-1.7cm,font=\color{black}]  {\small Accuracy};
  % \node[above of= img1, node distance=0cm, yshift=-2.5cm,font=\color{black}]  {\small (a)};
  % \node[left of= img1, node distance=0cm, rotate=90, anchor=center,yshift=2.2cm,font=\color{black}] { Final KL divergence };
\end{tikzpicture}\columnbreak
\hspace{20mm}
\begin{tikzpicture}
  \node (img1)  {\includegraphics[width=0.3000\textwidth]{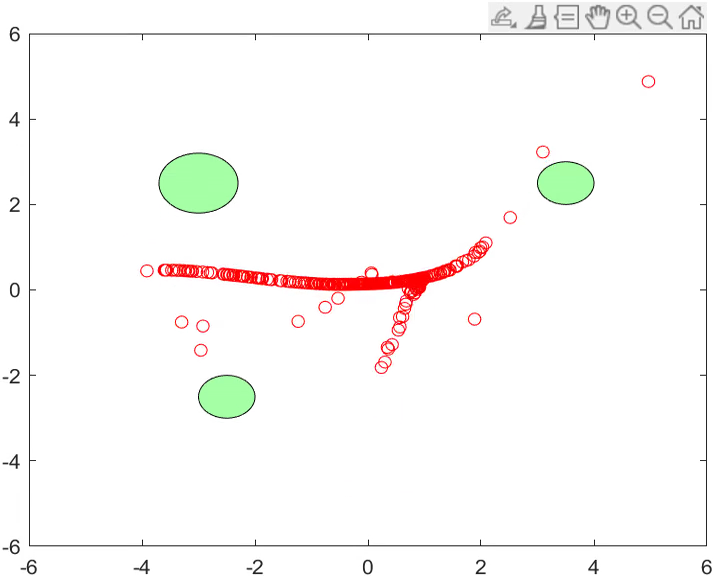}};
  \node[above of= img1, node distance=-0.5cm, yshift=-2.1cm,font=\color{black}]  {Final density}; 
  % \node[below of= img1, node distance=0cm, yshift=-1.7cm,font=\color{black}]  {\small $\log$ FPR};
  % \node[above of= img1, node distance=0cm, yshift=-2.5cm,font=\color{black}]  {(b)};
  % \node[left of= img1, node distance=0cm, rotate=90, anchor=center,yshift=2.8cm,font=\color{black}] {Cost};
\end{tikzpicture}
\end{multicols}
\vspace*{-2em}
\caption{Unicycle robots navigating a non-convex environment to reach the standard Gaussian.}
\label{fig:unicycle-obst}
\end{figure}

\begin{comment}
\begin{figure}
\centering
\subfloat[$n=50$]{
\includegraphics[width=0.25\textwidth]{figs/obst-init.png}%
}%
\hspace{1em}
\subfloat{
\includegraphics[width=0.25\textwidth]{figs/obst-final.png}%
}
 \caption{Unicycle robots navigating a non-convex environment to reach the standard Gaussian.}
\label{fig:unicycle-obst}
\end{figure}
\end{comment}

% See figure \ref{fig}
% \kecomment{Todo:figure here}

\subsection{Related work}
% \dgcomment{We need to cite OT and density control papers and state their shortcomings. }
% \dgcomment{As you state below, DDPMs haven't been used in feedback control directly. They have been used in trajectory planning and RL. Trajectory planning papers use a low-level stabilizing controller and use diffusion only for trajectory planning. RL operates in a different framework where the diffusion is on both action and state space.}

\textbf{Diffusion based planning.} While there has been some work on using DDPMs for planning and control, existing approaches are conceptually different from the way we exploit DDPMs. For example, \cite{janner2022planning} proposes using DDPM to construct a generative model of a system as a probability density on the set of trajectories on the state space. The work \cite{chi2023diffusion} improves on this by learning the action policy instead of the joint density on the action and state space, thus lowering the dimensionallity of the problem. Several extensions of diffusion based planning have been proposed for long horizon planning \cite{li2023hierarchical}, learning cost functions for robotic arms \cite{urain2023se}, multi-task reinforment learning \cite{he2023diffusion} and imitation learning \cite{hegde2023generating,xian2023chaineddiffuser}.  A drawback of the previous approaches is that if $A$ and $S$ represent the action space and the state space, respectively, the noising process is performed on the entire trajectory space $A^N$, $S^N$ or $A^N \times  S^N$, where $N$ is the control horizon. Differently from these approaches, we perform the noising process in the state space $S$ by treating the control system as the reverse or densoising process itself. This idea leads to a massive reduction in the dimensionality of the problem and a drastic increase in the applicability and performance of the proposed method. 

\textbf{Probability density control.} The problem of control design from the perspective of generative modeling also arises in the control theory literature in the context of probabilistic density control. These works include control design using the Perron-Frobenius operators \cite{vaidya2010nonlinear,buzhardt2023controlled}, convexication of control problems through probabilistic lifting \cite{hernandez1996linear,majumdar2014convex}, optimal transport of controlled dynamical systems \cite{agrachev2009optimal,hindawi2011mass,chen2016optimal,elamvazhuthi2023dynamical}, and probabilistic control design \cite{karny1996towards}. When compared to DDPMs and our proposed control strategy, these methods have the same drawbacks as continuous normalizing flows and  transport based methods for generative modeling.

% \margin{I rewrote the controbutions. Please check and modify as you see appropriate}

\subsection{Contributions}
The main contributions of this paper are as follows. First, we propose a novel DDPM-based approach to design feedback controllers for nonlinear systems, that is, to transfer in finite time the system from an initial density of initial states to a desired density of target states. This introduces a novel generative modeling approach to control design. Second, we provide theoretical guarantees for the feasibility and effectiveness of the proposed method. In particular, for control affine systems without drift, we prove the existence of a controller that tracks exactly the trajectory of the density of the system states generated by the stochastic forward process (see Theorem \ref{thm:tackdens})). Further, we show how such controller can be designed by training a neural network. Third, we address the computational issues due to the fact that, in practice, the trajectory of the density of the state can only be measured at finite set of times and only estimated using a finite set of state samples. Finally, we validate our method for controller design through a number of examples. We conduct further experiments to demonstrate the ability of our DDPM-based control approach for cases beyond our theoretical results.
% \kecomment{Up is rewrite of below. Check}

% i) We present a DDPM based approach to design feedback controllers for nonlinear control systems, to transfer the system from a noise density to a target probability density, thus providing a generative modeling approach to design feedback controllers.

% ii) We provide a theoretical justification for the DDPM based feedback control design approach, for the case of control affine systems without drift, which include models for systems such as wheeled robots and robotic arms. We prove if the forward process is the Brownian motion, confined to a compact domain, DDPMs can incorporate control system constraints into the reverse process (Theorem \ref{thm:tackdens}).

% \dgcomment{need to rephrase (iii)}
% {\color{red}
% iii) We also address the practical case where we are only provided samples from the initial and final densities. We also consider the case where the state cannot be measured throughout the time horizon of diffusion, but only at specific time instants. We provide a method to identify the controller that drives the system to the target density using only samples of the state at specific time instants.
% }

% iv) We validate our method for controller design on a number of examples .

% \kecomment{iv) needs to be compeleted}.

% \newpage

\section{Problem setting and preliminary notions}
Consider the driftless, affine, nonlinear control system
\begin{align}
\label{eq:ctrsys}
    \dot x = g(x,u) = \sum_{i=1}^m  g_i(x) u_i,  %f(x,u) .
\end{align}
where $x \in \mathbb{R}^d$ denotes the state  of the system, $u_i$ the $i$-th control input, and $g_i \in C^{\infty}(\mathbb{R}^d;\mathbb{R}^d)$ are smooth vector fields. The model \eqref{eq:ctrsys} is commonly used in robotics to describe the dynamics of ground vehicles, underwater robots, manipulators, and other non-holonomic systems~\cite{JCL:2012}. In this paper we aim to solve a finite-time density control problem, that is, to find a control policy $u = \pi(t,x)$ such that the density of the state $x(T)$ is $p_\text{target}$ when the density of the initial state is $p_\text{initial}$, for a given control horizon $T \in \mathbb{R}_{>0}$. To this aim, we propose a methodology that relies on recent DDPMs, where the control policy is given by a neural network whose output is a control policy that allows the system \eqref{eq:ctrsys} to track the trajectories generated by reference stochastic diffusion process that ensures $x(T) \sim p_\text{target}$. To further clarify our approach, we next review denoising diffusion models through the lens of stochastic differential equations as presented in \cite{song2020score}, and formally state the control problem of interest and its applications.

\subsection{Denoising Diffusion Probabilistic model}
\label{sec:bgrnd}
Denoising Diffusion Probabilistic Model (DDPM) is a generative modeling technique that learns to sample from an unknown data density $p_\text{target}$ by (i) transforming it into a known  density $p_\text{initial}$ from which one can easily sample (such density has been referred to as the \emph{noise density}), (ii) sampling from it, and finally (iii) reversing the transformation.  This is done using two stochastic differential equations. First, the {\it forward process} is a stochastic differential equation of the form
\begin{align}\label{eq:sde}
\textup{d} X^\text{f} = V(X^\text{f}) + \sqrt{2} \, \textup{d} W + \textup{d} Z,
\end{align}
where $X_0$ has density $p_\text{target}$. In \eqref{eq:sde}, $W$ denotes the standard Brownian motion and $Z$ is a stochastic process that ensures that the state remains confined to a domain of interest $\Omega$.\footnote{The normalization constant $\sqrt{2}$ simplifies Equation~\eqref{eq:fwdpdf}.} The vector field $\map{V}{\mathbb{R}^d}{\mathbb{R}^d}$ is chosen such that the density of the random variable $X$ converges to $p_\text{initial}$. For example, if $V = \nabla \log p_{\text{initial}}= \frac{\nabla p_{\text{initial}}}{p_{\text{initial}}}$ then $\lim_{t \rightarrow \infty} p_t = p_{\text{initial}}$ \cite{bakry2014analysis}. Such convergence is guaranteed by the Fokker Planck equation that governs the evolution of the density $p$ of the state  $X$ of~\eqref{eq:sde}: 
\begin{align}
\label{eq:fwdpdf}
\partial_t p = \Delta p - \text{div}(V(x) p) = \text{div}  ([\nabla \log p -V(x)] p) . % \\ \nonumber
%.
\end{align}
where  $\Delta : = \sum_{i=1}^d \partial_{x^2_i}$ and $\text{div} : = \sum_{i=1}^d \partial_{x_i}$ denote the Laplacian and the  divergence operator, and $p_0 = p_\text{target}$.

% The vector field $V:\mathbb{R}^n \rightarrow \mathbb{R}^n$ is chosen such that the probability density of the random variable $X(t)$ converges to the desired density $p_{\text{noise}}$, from which one can easily sample, and is referred to as the {\it noise density}. For example, if one takes $V(x) = \nabla \log p_{\text{noise}}= \frac{\nabla p_{\text{noise}}}{p_{\text{noise}}}$ then $\lim_{t \rightarrow \infty} p_t = p_{\text{noise}}$ \cite{bakry2014analysis}.
 
When $\Omega$ is a strict subset of $\mathbb{R}^d$, this equation is additionally supplemented by a boundary condition, known as the {\it zero flux} boundary condition
\begin{equation}
\vec{n}(x)\cdot ( \nabla p _t(x) - V(x)) = 0 ~~~ \mbox{on} ~~~ \partial \Omega ,
\end{equation}
where $\vec{n}(x)$ is the unit vector normal to the boundary $\partial \Omega$ of the domain $\Omega$.
This boundary condition ensures that $\int_{\mathbb{R}^d} p_t(x)dx = 1$ for all $t\geq 0$. An advantage of considering the situation of bounded domain is that one can choose $V\equiv 0$ and the noise density can be taken to be the uniform density on $\Omega$. This property will be useful when solving our control problem.

%We can rewrite equation \eqref{eq:fwdpdf} as
%\begin{align}
%\partial_t p^{\textup{f}} = \text{div}  ([\nabla \log p^{\textup{f}} +V(x)] p^{\textup{f}}) 
%\end{align}

The second part of the DDPM is the \emph{reverse process}, which aims
to transport the density from $p_\text{initial}$ back to $p_\text{target}$. There are multiple possible choices of the reverse process, including the {\it probabilistic flow ODE} given by
\begin{align}\label{eq: reverse process}
\textup{d}X^{\textup{r}} = -\nabla \log p_{T-t}(X^{\textup{r}})\textup{d}t + V (X^{\textup{r}})\textup{d}t, % \nonumber \\
%X_0 \sim p_{\text{noise}}
\end{align}
where $X^{\textup{r}}_0$ has density $p_\text{initial}$ and $T$ is the horizon of the reverse process. 
%Fixing $T>0$, one can first sample from $p_{\text{noise}}$ and run the reverse process that approximately has the same probability density as $p_{T-t}$. 
%There are multiple possible choices of the reverse process. One candidate choice, is the {\it probabilistic flow ODE} given by,
%\begin{align} 
%dX^{\textup{r}} = -\nabla \log p^{\textup{f}}_{T-t}dt + V (X^{\textup{r}})dt %\nonumber \\
%X_0 \sim p_{\text{noise}}
%\end{align}
The evolution of the density $p^{\textup{r}}$ of the reverse process \eqref{eq: reverse process} is
\begin{align}
\partial_t p^{\textup{r}} = \text{div}  (-[\nabla \log p_{T-t} +V(x)] p^{\textup{r}}) .
\end{align}
In the ideal case, $p_{\text{initial}} = p_T$ and, $p_{T-t} = p^{\textup{r}}_t$ for all $t \in [0,T]$. 
%emphasizing the terminology of referring to \eqref{} and \eqref{} as the forward and reverse processes, respectively.
After simulating the forward process, one can learn the \emph{score} $\nabla \log p_{T-t}$ to run the reverse process to effectively sample from $p_{\text{target}}$.
%, using the {\it score} $\nabla \log p^{\textup{f}}_{T-t}$.
In practice, $p^{\textup{f}}_T \approx p_{\text{initial}}$, and hence $p^{\textup{r}}_T  \approx p_0 = p_{\text{initial}}$, and one does not have complete information about the score. Usually, a neural network $\texttt{NN}(t,x,\theta)$ is used to approximate the score by solving the optimization problem
% \dgcomment{do we need $\cdot$ in eqn (7)?}
\begin{align} 
\min_{\theta} & \int_0^T \mathbb{E}_{p_{T-t}} |\texttt{NN}(t,\cdot,\theta) -\nabla \log p_{T-t} |^2\textup{d}t 
\end{align}
This objective ensures that the solution $p^\theta_t$ of the equation 
\begin{align}
\partial_t p^\theta = \text{div}  ( [ \texttt{NN}(t,x,\theta) +V (x)] p^\theta),
\end{align}
where $p^\theta_0 = p_{\text{initial}}$,  is close to $p_t$ so that we can sample from $p_{\text{target}}$ by running the reverse ODE,
\begin{align} 
\textup{d}X^\theta = \texttt{NN}(t,X^\theta,\theta)-  V (X^\theta)\textup{d}t 
\end{align}
such that $X_0$ is sampled from $p_{\text{noise}}$. 
\subsection{DDPM Control Problem}
% \dgcomment{change density dynamics in figure}
\begin{figure*}
\label{fig:control-density-probs}
\centering
% \begin{tikzpicture}
  \includegraphics[width=0.95\textwidth]{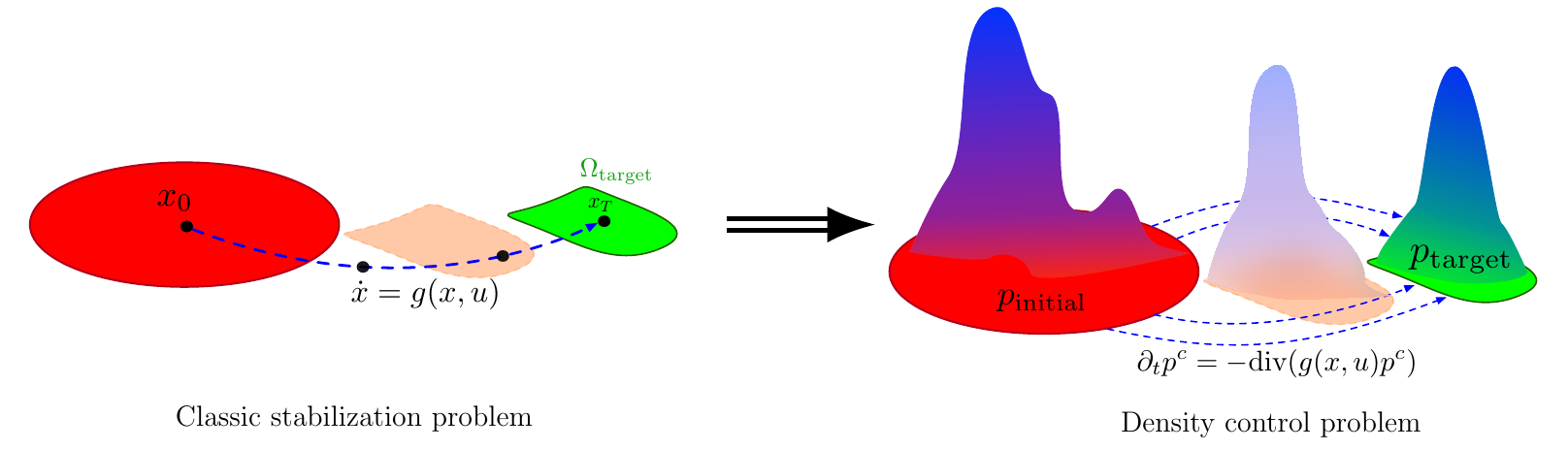}
% \end{tikzpicture}
\caption{Reformulation from the classical control problem to the density control problem. }
\end{figure*}
We now formally state the DDPM based feedback control problem. 
Toward this end, we will need a description of how the density of the  solution $x$ of equation \eqref{eq:ctrsys} which evolves. This is known to be given by the Liouville equation or the continuity equation \cite{santambrogio2015optimal},
   \begin{align}
\label{eq:fwdpdf_cs}
\partial_t p^{\textup{c}} = -\textup{div}(g(x,\pi(t,x)) p^{\textup{c}}), 
\end{align}
for a given initial density $p^{\textup{c}}_0 = p_0$.
\begin{problem}
Does there exist a feedback control policy $u = \pi(t,x)$, such that the density of the controlled state $p^{\textup{c}} = p_{\text{target}}$, and $p^{\textup{c}}_t = p^{\textup{f}}_{T-t}$ for all $t\in [0,T]$? If yes, identify such a feedback control policy $\pi(t,x)$. \label{prob:DDPM-ctrl-prob}
\end{problem}
The ideal control problem we are interested in is one of stabilizing the trajectories of the nonlinear system \eqref{eq:ctrsys} to a set of desired states, possibly the origin. Ideally, we are interested in designing a control policy $u = \pi(t,x)$ for \eqref{eq:ctrsys} such that the state $x$ reaches a predefined set $\Omega_{\text{target}} \subset \mathbb{R}^d$. This is a fundamental control problem with numerous applications in robotics, among other fields, that has received considerable attention over the years. The difficulty of such problem stems from the nonlinearity of the dynamics \eqref{eq:ctrsys}, which typically allow only for ad-hoc and often limited solution methods~\cite{AS:1985}. Inspired by the recent aforementioned successes of diffusion-based generative techniques, we reformulate the above deterministic control problem as a density control problem for \eqref{eq:ctrsys}. The density control formulation allows for probabilistic control of the population of states, instead of individual trajectories.
% The desired control policy aims to ensure a desired trajectory of the density of the state of the system, rather than individual trajectories of states. 
To be specific, let $p_\text{initial}$ be the density of possible initial states of the system, and let $p_\text{target}$ be the density of the desired final states. Let $p^{\textup{c}}_t$ be the density at time $t$ of the state of \eqref{eq:ctrsys} when $p^{\textup{c}}_0= p_\text{initial}$, and let $p_\text{desired}$ be a desired trajectory of the density of the state such that $p_\text{desired}= p_{\text{target}}$. Then, the DDPM control problem considered in this paper is to design a control policy $u = \pi(t,x)$ such that $p^{\textup{c}} = p_\text{desired}$. Figure~\ref{fig:control-density-probs} illustrates the reformulation of the two problems.

We conclude this section with some remarks. First, the DDPM control problem includes the deterministic stabilization problem as a special case. This can be seen by selecting $p^\text{desired}$ such that it is supported only on the domain of interest. Further, the classic stabilization problem is also a special case of the DDPM control problem, as the support of $p_\text{target}$ can be chosen as a singleton set $\Omega_\text{target}=\{x_\text{target}\}$. Second, the connection between the DDPM technique and the DDPM control problem is evident, as we can select the desired density trajectory $p^\text{desired}_t$ as the (time-reversed) trajectory generated by the forward process \eqref{eq:fwdpdf}. The nonlinear system~\eqref{eq:ctrsys} with feedback policy $\pi(t,x)$ then acts as the reverse process of the DDPM. Hence, our goal is to design the feedback policy such that the reverse policy tracks the (time-reversed) forward process. It is important to note that this control problem is much harder to solve than the DDPM problem for generative modeling. The control policy $\pi$ plays the role of the score, however there might not be a control policy $\pi(t,x)$ that realizes the score $\nabla \log p^{\textup{f}}_t(x)$ since not all directions of the state space might be admissible (this is the case in underactuated systems~\cite{FB-NEL-ADL:2000}).

\section{Algorithms for DDPM-based control }

In this section, we provide a systematic procedure to solve the DDPM control problem and discuss theoretical aspects of learning the controller $\pi(t,x)$ such that the density of the forward process can be tracked by states of the nonlinear system in reverse.
% \subsection{Algorithm}
Our DDPM based approach to address the probability density control problem is the following. Adapting the methodology introduced in Section \ref{sec:bgrnd}, we take \eqref{eq:ctrsys} to be the reverse process, retaining \eqref{eq:sde} as the forward process. The forward process provides a reference trajectory $p^{\textup{f}}_{T-t}$ in the set of probability densities such that $p^{\textup{f}}_T \approx p_{\text{initial}}$ and $p_0 = p_{\text{target}}$. Thus if there exists a controller $\pi(t,x)$ that can ensure system \eqref{eq:ctrsys} exactly tracks this density trajectory, we will have that $p_T^{\textup{c}} \approx p_{\text{target}}$. 
% \dgcomment{insert figure for density control with forward and reverse process here. }

% \begin{figure}[b]
% \label{fig:IniFin}
% \centering
% \includegraphics[scale=0.22,trim={0 1cm 0 1cm},clip]{Diffusion Schematic.pdf}%
% \caption{Schematic of DDPM based control framework.}
% \end{figure}

% A significant complication arises in this general control theoretic setting in comparison to the classical DDPM framework.   There might not be a control $\pi(t,x)$ such that $f(x,\pi(t,x)) = -\nabla \log p^{\textup{f}}_t(x) $, since not all directions in the state space might be admissible, thus making the score $\nabla \log p^{\textup{f}}_t(x)$ {\it not realizable} by the control system. Hence, we solve this control theoretic version of the problem by looking for a control $\pi(t,x)$ such that $ p^{\textup{c}} = p^{\textup{f}}_{T-t}$ for $t \in [0,T]$.
To identify the controller $\pi(t, x)$, we seek to the following minimization problem:
\begin{align}
\min_{\pi} \frac{1}{T}\int\limits_0^T \text{KL}(p^{\textup{c}} \big|p^{\textup{f}}_{T-t}) \ \textup{d}t. \label{eqn:cost-continuous}
\end{align} 
 subject to the dynamics \begin{align}\partial_t p^{\textup{c}} = -\text{div} (g(x,\texttt{NN}(t,x,\theta)) p^{\textup{c}}),
\end{align} 
where $p^{\textup{c}}_0 = p_{\text{initial}}$. 
Here $\text{KL}(p^{\textup{c}}\big|p^{\textup{f}})$ denotes the KL divergence between the density of the control system and the forward density. The KL divergence between any two densities $Q$ and $R$ defined on a set $\Omega$ is given by
\begin{align}
\text{KL}(Q\big|R) &= \int\limits_{x \in \Omega} Q(x) \log\left(\frac{Q(x)}{R(x)} \right) ~\textup{d}x.
\end{align}
Note that the cost function in problem~\eqref{eqn:cost-continuous} is ideal as it considers the KL-divergence over the whole horizon $[0,T]$ over the entire domain $\Omega$. In practice, the states of the system can only be measured at specific time instances. Hence, we can only estimate the KL divergence at specific instances of the horizon using finite samples from the domain. Therefore, we solve the following more practical problem:
\begin{align}
\min_{\pi} \frac{1}{N}\sum_{i=1}^N \widehat{\text{KL}}(p^{\textup{c}}_{t_i}\big|p^{\textup{f}}_{T-t_i}).\label{eqn:approx-prob}
\end{align}
Here, $\widehat{\text{KL}}$ denotes the estimated KL divergence between $p^{\textup{c}}_{t}$ and $p^{\textup{f}}_{T-t}$, based on their respective samples. The time instances when the state is measured is given by $0\leq t_1 < t_2 < \dots < t_N = T$. Further, we use $M$ samples to estimate the KL-divergence and the samples are constrained to the dynamics~\eqref{eq:ctrsys} in the reverse process. We provide a method to numerically estimate the KL divergence directly using $M$ samples in Section~\ref{sec:KLD}. The procedure to solve problem~\eqref{eqn:approx-prob} is given in Algorithm~\ref{algo:DDPM-FL}.
% \dgcomment{need to insert algorithm here.}
\begin{algorithm}[tbh] 
\caption{Learning the controller $\pi$}
\begin{algorithmic}\label{algo:DDPM-FL}
\STATE\textbf{Data:} $M$ samples from $p_\text{target}$, System~\eqref{eq:ctrsys}. \\ 
\STATE\textbf{Initialize:} Time instances for measuring state: $\{t_i\}_{i=1}^N$, neural network $\texttt{NN}(t,x,\theta)$, number of training epochs $E$. \\
% \SetAlgoLined
\STATE 1. Run forward process for $M$ samples for horizon $[0,T]$:
\vspace*{-0.4em}
$$
\textup{d} X^{\textup{f}} = V(X^{\textup{f}}) + \sqrt{2} \, \textup{d} W + \textup{d} Z. $$ 
\STATE \vspace*{-0.1em} 2. Measure $X^{\textup{f}}$ for all $M$ samples at instances $\{t_i\}_{i=1}^N$.
\STATE 3. \textbf{For} $k = 1:E$, \{
\STATE 4. \quad Run reverse process for $M$ samples as:
\vspace*{-0.4em}
$$\dot{x} = g(x,\texttt{NN}(t,x,\theta)).$$
\STATE \vspace*{-0.2em} 5. \quad Measure state $x$ at time instances $\{t_i\}_{i=1}^N$.
\STATE 6. \quad Estimate KL-divergence $\widehat{\text{KL}}(p^{\textup{c}}_{t_i}|p^{\textup{f}}_{T-t_i})$ for all $t_i$.
\STATE 7. \quad Optimize controller $\texttt{NN}_{\theta}$ by minimizing cost~\eqref{eqn:approx-prob}.\ \ \}\\
\end{algorithmic}
\end{algorithm}
% \dgcomment{remove $\beta$ and re-word }
\begin{remark}{\textbf{(Practical aspects of Algorithm~\ref{algo:DDPM-FL})}} Firstly, the choice of $V(x)$ in the forward process determines the final density of the $M$ data samples. Particularly, a choice of $V(x) = 0$ ensures that $p_{\textup{initial}}$ is the uniform distribution in the domain $\Omega$. The choice $V(x) = - k x$, ensures $p_{\textup{initial}}$ is the standard Gaussian distribution if $T$ is sufficiently large. Here, $k$ is a positive scalar. In this work, we will particularly focus on the case where $V \equiv 0$, that is, we seek to uniformly sample from everywhere in the domain and reach $p_{\textup{target}}$ in the reverse process. Secondly, scaling the noise $\textup{d}W$ in the forward process~\eqref{eq:fwdpdf} determines how fast $p_\textup{initial}$ is reached. 
Conversely, the noise scale determines the control effort $\|\pi(x,t)\|$ in the reverse process. For smoother and lower control effort, it is essential to use smaller noise scales. Lastly, given the choice of $V(x)$ and the noise scale, the length of the horizon $[0,T]$ determines whether the forward process reaches $p_\textup{initial}$ and whether the dynamical system can reach $p_\textup{target}$. Dynamical systems with actuator constraints require larger $T$ and smaller noise scale.
\end{remark}

\begin{remark}{\textbf{(Choices of different cost functions for learning controller $\pi$)}}. The KL divergence serves as a good metric in the cost~\eqref{eqn:approx-prob} to learn the controller $\pi$. The KL-divergence can be approximated in a numerically convenient method as shown in Section~\ref{sec:KLD}. However, we are not restricted to using the KL-divergence based cost function. The Wassertein distance is particularly difficult to compute in high dimensions. Another intuitive choice of the cost function could be the difference between moments of the distributions $p^{\textup{c}}_{t_i}$ and $p^{\textup{f}}_{T-t_i}$. Although estimating moments is numerically efficient, different densities can possess the same moments and hence, may not serve as a good cost function generally. 
\end{remark}

\section{Theoretical Analysis and Results}
\label{sec:theres}

% {\color{blue}This below should be stated as a formal result in the theoretical section. With its proof}

% \kecomment{I think this should be left here for the informal discussion. We can refer to the theorem for the precise statement.}
% \dgcomment{Not sure if we need this informal discussion of the theory. Theorem~\ref{thm:tackdens} reads well as a statement, so we might not need this discussion. Maybe we can use this subsection to discuss practicality of the algorithm, choice of different parameters of the diffusion model such as noise scale, time horizon, etc?  }

A natural question that arises from a theoretical point of view is the well-posedness of this optimization problem~\eqref{eqn:cost-continuous}. Particularly, the nonlinear dynamics of system~\eqref{eq:ctrsys} play a crucial role in the feasibility of problem~\eqref{eqn:cost-continuous}. In typical DDPM problems for generative modeling, noise can be added in all directions for both the forward and reverse processes. Therefore, score-matching techniques can be employed to learn the reverse process. However, for the feedback control problem in consideration, if the dynamical system is not fully actuated, that is $d=m$, then the feasibility of problem~\eqref{eqn:cost-continuous} is not clear. Here, we address the case when such a controller can be obtained. We reserve a formal theoretical result for Section~\ref{sec:theres}.

For system~\eqref{eq:ctrsys}, consider $\map{g_i}{\Omega}{\mathbb{R}^d}$, and an observable of the system $\map{h}{\Omega}{\mathbb{R}}$. With input $u_i \neq 0$ and $u_j = 0$ for all $j\neq i$, the time derivative of the function $h$ can be represented in terms of a differential operator $Y_i$ as,
\begin{align}
\dot{h}(x(t)) = Y_i h(x(t) &= \sum\limits_{j=1}^d g_{i}^{j}(x(t)) \ \partial_{x_j} h(x(t)),
\end{align}
where $g^j_i$ represents the $j^{th}$ element of $g_i(x)$. The differential operator $Y_i$ has an associated adjoint $Y_i^*$~\cite{AL-MMC:2013} which defines the evolution of the density, allowing us to rewrite Equation \eqref{eq:fwdpdf_cs} as 
\begin{align}
\label{eq:ctctty}
\partial_t p^{\textup{c}}  = \sum_{i=1}^m  Y_i^*(u_i p^{\textup{c}}) = \sum\limits_{j=1}^m \sum\limits_{i=1}^d-\partial_{x_j} (u_i g_{i}^j(x)\ p^{\textup{c}}).  
\end{align}
We will show in this section  the existence of a controller $\pi(t,x)$ that ensures that the system tracks the forward process in reverse is contingent on the condition that the operator $\sum_{i=1}^m Y_i^* Y_i$ is invertible.  Moreover, the invertibility of the operator is a condition on the controllability of system~\eqref{eq:ctrsys}. With full knowledge of the system and the densities $p_\text{initial}$ and $p_\text{target}$, such a controller can be identified.

In order to state our main result, we will need to define some mathematical notions that will be used in this section.  We define $L^2(\Omega)$ as the space of square integrable functions over $\Omega$, where $\Omega \subset \mathbb{R}^d$ is an open, bounded and connected subset. The set of continuous functions $p$ for which $p_t$ lies in $L^2(\Omega)$ will be referred to using $C([0,1];L^2(\Omega))$. Let $\mathcal{P}_2(\mathbb{R}^d)$ denote the set of Borel probability measures on $\mathbb{R}^d$ with finite second moment: $\int_{\Omega} |x|^2d\mu(x)~<~\infty$. For a given Borel map $T : \mathbb{R}^d \rightarrow  \mathbb{R}^d$ we will denote by $T_{\#}$ the corresponding pushforward map, which maps any measure $\mu $ to a measure $T_{\#}\mu$, where $T_{\#}\mu$ is the measure defined by
	\begin{equation}
	(T_\# \mu)(B) = \mu(T^{-1}(B)), 
	\end{equation}
	for all Borel measurable sets $B \subseteq \mathbb{R}^d$. For $\mu,\nu \in \mathcal{P}_2(\mathbb{R}^d)$, we denote the set of transport plans from $\mu$ to $\nu$ by
\begin{equation}
\Gamma(\mu,\nu):=\{\gamma \in \mathcal{P}(\mathbb{R}^d \times \mathbb{R}^d) | \pi^1_{\#}\gamma = \mu, \pi^2_{\#}\gamma = \nu \},
\end{equation}
where $\pi^i:\mathbb{R}^d \times \mathbb{R}^d \rightarrow \mathbb{R}^d$ are the projections on to the $i$th coordinates, respectively.
We will define the $2-$Wasserstein distance between two probability measures $\mu,\nu$
as the following 
\begin{equation}
W_2(\mu,\nu) = \min_{\gamma \in \Gamma(\mu,\nu)} \Bigg (\int_{\mathbb{R}^d \times \mathbb{R}^d} |x-y|^2d\gamma(x,y)\Bigg )^{1/2}.
\end{equation}

	Let $\mathcal{V} = \lbrace g_1,...,g_m \rbrace$, $m \leq d$, be a collection of smooth vector fields $g_i : \mathbb{R}^d \rightarrow \mathbb{R}^d$. Let $[f,g]$ denote the Lie bracket operation between two vector fields $f: \mathbb{R}^d \rightarrow \mathbb{R}^d$ and $g: \mathbb{R}^d \rightarrow \mathbb{R}^d$, given by
where $\partial_{i}$ denotes partial derivative with respect to coordinate $i$.
 \begin{equation}
 [f,g]_i = \sum_{j=1}^d f^j \partial_{x_j} g^i - g^j \partial_{x_j} f^i.
 \end{equation}
 
 We define $\mathcal{V}^0 =\mathcal{V}$. For each $i \in \mathbb{Z}_+$, we define in an iterative manner the set of vector fields $\mathcal{V}^i = \lbrace [g, h]; ~g \in \mathcal{V}, ~h \in \mathcal{V}^{j-1}, ~j=1,...,i \rbrace$.  We will assume that the collection of vector fields $\mathcal{V}$ satisfies following condition the {\it Chow-Rashevsky} condition \cite{agrachev2019comprehensive} (also known as {\it H\"{o}rmander's condition} \cite{bramanti2014invitation})
 
 \begin{assumption}
 \label{asmp2}
 \textbf{(Controllability})
     The Lie algebra generated by the vector fields $\mathcal{V}$, given by $\cup_{i =0}^r \mathcal{V}^i$, has rank $N$, for sufficiently large $r$.
 \end{assumption} 

We will need another assumption on the regularity of $\Omega$. Towards this end an {\it admissible curve} $\gamma:[0,1] \rightarrow \Omega $ connecting two points $\x ,\y \in \Omega$ is a Lipschitz curve in $\Omega$ for which there exist essentially bounded functions $u^i:[0,T]\rightarrow \mathbb{R}$ such that $\gamma$ is a solution of \eqref{eq:ctrsys} with $\gamma(0)=x$ and $\gamma(1) = y$.

\begin{definition}
 \label{def:epsdel}
	The domain $\Omega \subset \mathbb{R}^d$ is said to be  \textbf{non-characteristic} if for every $x \in \partial \Omega$, there exists a admissible curve $\gamma(t)$ such $\gamma(0) = x$ and $\gamma(t)$ is not tangential to $\partial \Omega$ at $x$.
\end{definition}
This definition imposes a regularity on the domain $\Omega$, which will be needed to apply the results of \cite{elamvazhuthi2023density} to conclude the invertibility of the operator $\sum_{i=1}^m Y_i^*Y_i$.

\begin{assumption}
\label{asmp1}
\textbf{(Boundary regularity)}
We will make the following assumption on the boundary of the domain $\Omega$.
\begin{enumerate}
\item The domain $\Omega$ has a $C^2$ boundary $\partial \Omega$.
\item The domain $\Omega$ is non-characteristic in the sense of Definition \ref{def:epsdel}.
\end{enumerate}
\end{assumption}

We will say that $p \in L^2(\Omega)$ is a probability density, if $\int_{\Omega}p(x)dx =1 $ and $p$ is non-negative almost everywhere on $\Omega$.

Given these definitions we will show that we can a find a feedback controller $u_i = \pi_i(t,x)$ such that the dynamical system tracks the forward process in reverse. 

\begin{comment}
We will present an informal discussion as to why this is true. Let us suppose that $p^{\textup{c}} >0$ on $\Omega$ for all $t \in [0,T]$. Let us say the controller $\pi(t,x)$ takes the form, 
\begin{align}
\pi_i(t,x) = \frac{Y_i \phi(t,x)}{p^{\textup{c}}}, \label{eqn:controller-form}
\end{align}
for some function $\map{\phi}{[0,T] \times \Omega}{\mathbb{R}}$, for all $i$. If we assume $p^{\textup{c}} > 0$ everywhere in the domain $\Omega$, our task is to now identify the function $\phi(t,x)$. To this end, let us first substitute $\pi_i$~\eqref{eqn:controller-form} in the continuity equation~\eqref{eq:ctctty} as, 
\begin{align}
\partial_t p^{\textup{c}} = \sum\limits_{i=1}^m Y_i^* Y_i \phi(t,x).
\end{align}
To find the function $\phi(t,x)$, we need to solve the equation, 
\begin{align}
\phi(t,x) &= \left(\sum_{i=1}^m Y_i^* Y_i \right)^{-1} \partial_t p_t^{\textup{c}}. \label{eqn:control-soln}
\end{align}

Note that the equation~\eqref{eqn:control-soln} can be solved as long as the operator $\left(\sum_{i=1}^m Y_i^* Y_i \right)$ is \emph{invertible}. With full knowledge of the system dynamics, and the initial and target densities, the controller can be identified as,
\begin{align}
\pi_i(t,x) = \frac{Y_i \phi(t,x)}{p^{\textup{c}}} = \frac{1}{p^{\textup{c}}} \sum\limits_{j=1}^{d} g_i^j(x) \ \partial_{x_j} \phi(t,x).
\end{align}
\end{comment}

First, we have the following result that the control system can track any arbitrary trajectory on the set of probability densities, provided it satisfies some smoothness assumption. The complete proof is provided in Appendix \ref{sec:app}.
\begin{lemma}[\textbf{Exact tracking of positive densities}]
\label{lem:extrafull}
Given Assumption \ref{asmp2} and \ref{asmp1}, suppose $p^{\text{ref}} \in C([0,T];L^2(\Omega))$, $\partial_t p^{\text{ref}} \in C([0,T];L^2(\Omega))$, $ p^{\text{ref}}_t $ is a probability density and  $  p_t > 0$,  for all $t\geq 0$. 
 
 Then there exists a control law $\pi^i$ such that a solution $p^{\textup{c}}$ of the \eqref{eq:ctctty} satisfies 
\[p^{\textup{c}} = p^{\text{ref}}_t \text{ for all } t\in [0,T]\] 
\end{lemma}

Given the previous lemma on tracking probability densities, we state the following theorem on tracking the reversed trajectories of the forward process.

\begin{theorem}[\textbf{Feedback control by tracking reverse process}]
\label{thm:tackdens}
Given Assumption \ref{asmp2} and \ref{asmp1}, suppose $p_{\text{target}}  \in L^2(\Omega)$ is a probability density. Let $V \equiv 0$. Then there exist control laws $\pi^i(t,x)$  and a solution of \eqref{eq:ctctty} that exactly tracks the solution of the forward process: $p^{\textup{c}}_t = p^{\textup{f}}_{T-t}$ for all $t \in [0,T)$. Moreover, we have,
\[ \lim_{t \rightarrow T^{-}} \|p^{\textup{c}}_t- p_{\text{target}}\|_2 = 0 \]
Suppose, instead $p_{\text{target}} = \delta_{x_0}$ for some $x_0 \in \Omega$. Then
		\[ \lim_{t \rightarrow T^{-}}  W_2(p^{\textup{c}}_t, p_{\text{target}}) = 0 \]
\end{theorem}

\section{Numerical Studies}
In this section, we elucidate the KL divergence approximation method and validate the DDPM based feedback control approach using numerical results.
\subsection{KL-divergence approximation}
\label{sec:KLD}
The density of the forward process $p^{\textup{f}}_t$ is approximated by simulating the $M$-particle system~\eqref{eq:sde} with $V(x)=0$. This yields an approximation of the density $p^{\textup{f}}_t\approx \frac{1}{N}\sum_{j=1}^N \delta_{X^j_t})$. Similarly, the density of the reverse process is given by simulating the control system~\eqref{eq:ctrsys}. 
% \begin{align} 
% dX^j =  \sqrt{2} dW^j +dZ\\
% X^j_0 \sim p_{\text{target}} \nonumber
% \end{align}
% Then we get an approximation of the density $p^{\textup{f}}_t\approx  \frac{1}{N}\sum_{j=1}^N \delta_{X^j_t})$. 
% Similarly, we have the approximating reverse process
% \begin{align} 
% \label{eq:Np sys}
% dX^j =  \sum_{i=1}^m s_i(t,X^j_t,\theta^i)g^i(X^j_t )+d\psi(t)\\
% X^j_0 \sim p_{\text{noise}} \nonumber
% \end{align}
% where we recall that $p_{\text{noise}}$ is the uniform density on $\Omega$.

In order to compute the KL divergence between densities, we use the kernel approximation of the densities by regularizing the KL divergence. This method of regularizing the KL divergence for computing KL divergence between particles is due to \cite{carrillo2019blob}. Here, a method for regularizing class of functionals on probability densities was introduced to develop a particle method to simulate solutions of certain nonlinear PDEs. 
Suppose $K :\mathbb{R}^d \rightarrow \mathbb{R}_{\geq 0}$ is a kernel function such that $\int_{\mathbb{R}^d} K(x)dx = 1$. Using this function, we define $K_{\delta} = \frac{1}{\delta^d}K(\frac{x}{\delta})$. Suppose $Q = \frac{1}{M}\sum_{i=1}^M \delta_{x_i}$, and $R = \frac{1}{M}\sum_{i=1}^M \delta_{y_i}$.  We can approximate $Q$ and $R$ using its kernel density estimate $K * Q = \frac{1}{M}\sum_{j=1}^N K(x-x^j)$ and $K * R = \frac{1}{M}\sum_{j=1}^M K(x-y^j)$. The KL-divergence is then approximated as 
\begin{align}
\widehat{\text{KL}}(Q|R)\ &\approx \int \sum_i^M  \log\left(\frac{K_{\delta} * Q(x)}{K_{\delta}*R(x)}\right)R(x)\textup{d}x \\
% &\approx \frac{1}{M} \sum_i^M  \log\left(\frac{K_{\delta} * Q(x_i)}{K_{\delta}*R(x_i)}\right) \\
& \approx \frac{1}{M} \sum_i^M \log \left(\frac{\sum_{j=1}^M K_{\delta}(x_i-x_j)}{\sum_{j=1}^M K_{\delta}(x_i-y_j)}\right)
\end{align}

% We use this approximation to numerically solve problem~\eqref{eqn:approx-prob}.
% \begin{align}
% \min_{\theta} \sum_{i=1}^N KL(p^{\textup{c}}_{t_i}|p^{\textup{f}}_{T-t_i}))\\
% \text{subject to \eqref{eq:Np sys}.} \nonumber
% \end{align}

\subsection{Experiments on dynamical systems}

We demonstrate the effectiveness of the proposed DDPM-based control method on nonlinear dynamical systems\footnote{\href{https://github.com/darshangm/diffusion-nonlinear-control/}{https://github.com/darshangm/diffusion-nonlinear-control}}. We characterize the effect of different parameters of the algorithm on the final KL-divergence between the $p_\text{target}$ and $p^{\textup{c}}$. We demonstrate our algorithm on three test beds.

\subsubsection{A five-dimensional bilinear system}
\begin{figure*}[tbh]
\begin{multicols}{3}
\hspace*{0.5cm}
\begin{tikzpicture}
  \node (img1)  {\includegraphics[width=0.2500\textwidth]{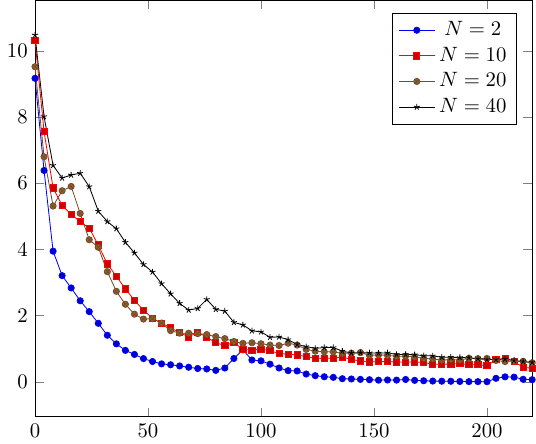}};
  \node[above of= img1, node distance=0cm, yshift=-2.2cm,font=\color{black}]  {Training iterations};  
  % \node[below of= img1, node distance=0cm, yshift=-1.7cm,font=\color{black}]  {\small Accuracy};
  \node[above of= img1, node distance=0cm, yshift=-2.6cm,font=\color{black}]  {\small (a)};
  \node[left of= img1, node distance=0cm, rotate=90, anchor=center,yshift=2.6cm,font=\color{black}] { Final Estimated KL divergence };
\end{tikzpicture}\columnbreak
\hspace*{0.7cm}
\begin{tikzpicture}
  \node (img1)  {\includegraphics[width=0.2500\textwidth]{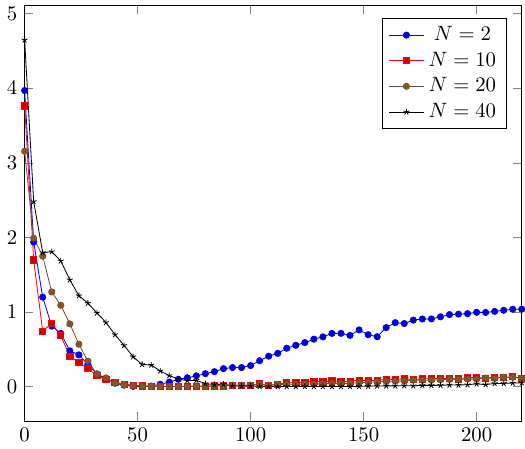}};
  \node[above of= img1, node distance=0cm, yshift=-2.2cm,font=\color{black}]  {Training iterations}; 
  % \node[below of= img1, node distance=0cm, yshift=-1.7cm,font=\color{black}]  {\small $\log$ FPR};
  \node[above of= img1, node distance=0cm, yshift=-2.6cm,font=\color{black}]  {(b)};
  % \node[left of= img1, node distance=0cm, rotate=90, anchor=center,yshift=2.6cm,font=\color{black}] {Control effort $\|\pi(t,x)\|$};
\end{tikzpicture}\columnbreak
% \hspace*{1cm}
\begin{tikzpicture}
  \node (img1)  {\includegraphics[width=0.2500\textwidth]{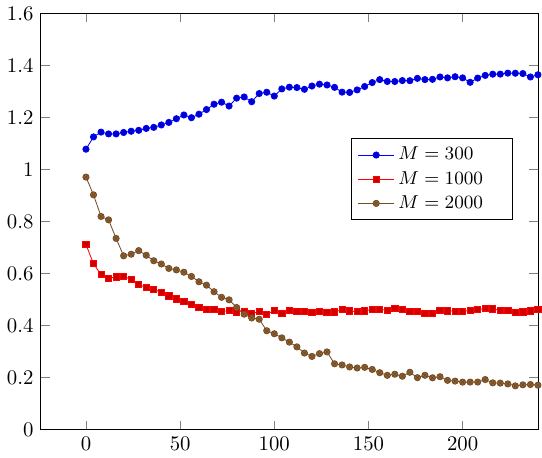}};
  \node[above of= img1, node distance=0cm, yshift=-2.2cm,font=\color{black}]  {Training iterations}; 
  % \node[below of= img1, node distance=0cm, yshift=-1.7cm,font=\color{black}]  {\small $\log$ FPR};
  \node[above of= img1, node distance=0cm, yshift=-2.6cm,font=\color{black}]  {(c)};
  \node[left of= img1, node distance=0cm, rotate=90, anchor=center,yshift=2.6cm,font=\color{black}] {Cost};
\end{tikzpicture}
\end{multicols}
\vspace*{-2em}
\caption{Experiments with a five-dimensional nonlinear system~\eqref{eqn:fived-sys}. (a) Final KL divergence for different number of measurement instances $N$ vs. training iterations with $M=2000$: validates Theorem~\ref{thm:tackdens} that neural network can ensure $p^{\textup{c}}=p_\text{target}$, (b) Final KL divergence vs. training iterations with $M=300$: shows that with fewer measurement instances, we need more training samples to approximate densities, (c) Cost \eqref{eqn:approx-prob} vs training iterations: depicts the sample complexity of the DDPM feedback control method. }
\label{fig:five-dim-sys}
\end{figure*}
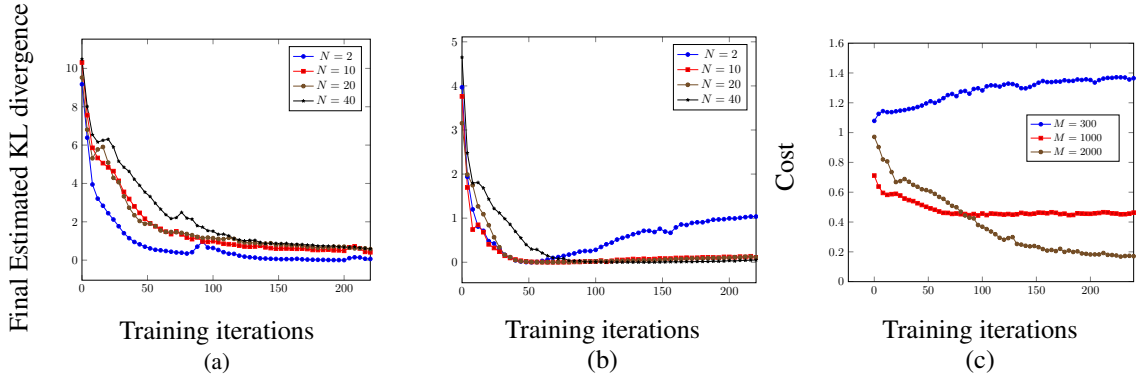

We consider the following five-dimensional driftless system.
% \begin{align}\label{eqn:fived-sys}
% \left[\begin{matrix}\dot{x}_1 \\ \dot{x}_2 \\ \dot{x}_3\\ \dot{x}_4\\ \dot{x}_5  \end{matrix}\right]&= \left[\begin{matrix}1 \\ 0 \\ x_2 \\ x_3 \\ x_4\end{matrix}\right] u_1+ \left[\begin{matrix}0 \\ 1 \\ 0 \\ 0 \\ 0\end{matrix}\right] u_2.
% \end{align}
\begin{align}\label{eqn:fived-sys}
\begin{aligned}
\dot{x}_1 &= u_1, \ \ \dot{x}_2 = u_2, \ \ \dot{x_3} = x_2 u_1, \dot{x}_4 &= x_3 u_1, \ \ \dot{x}_5 = x_4 u_1
\end{aligned}
\end{align}
 We sample $M$ data points from a Gaussian distribution $\mathcal{N}(0,0.2I) = p_{\text{target}}$. With $V(x)=0$ for the forward process, $p_\text{initial}$ is the uniform distribution defined over the domain $\Omega= (-4,4)^5$. The neural network used to estimate the controller $\texttt{NN}(t,x,\theta)$ has four hidden layers with 150, 500, 100 and 10 neurons in the hidden layers respectively. We first demonstrate the effect of the number of training samples $M$ on the final estimated KL-divergence $\widehat{\text{KL}}(p^{\textup{c}}\big|p^{\textup{f}}_0)$. To test the final KL divergence, we apply the learned controller on 2000 uniformly sampled  points in $\Omega$.
 
 Figure~\ref{fig:five-dim-sys}(a) shows the final estimated KL divergence between the density of the control system at time $T$ for increasing number of training iterations for the controller. Here, we use $M=2000$ training samples.  Figure~\ref{fig:five-dim-sys}(a) validates Theorem~\ref{thm:tackdens} by showing that the neural network is indeed able to control the system to minimize the KL divergence between the two densities for various number of measuring instances $N$.  However, in Figure~\ref{fig:five-dim-sys}(b), we see that when the number of training samples $M=300$, the controller is not able to reach the desired target density. This is because the controller does not have enough samples to estimate the different densities. Figure~\ref{fig:five-dim-sys}(c) depicts the cost~\eqref{eqn:approx-prob} for various training iterations when the number of measuring instances is $N=40$. It can be seen that when the number of training samples $M=300$, the controller performs poorly at minimizing the cost whereas when the training samples sufficiently cover the domain, that is, when $M=1000$ or $2000$, the neural network learns to reverse the forward diffusion. 
 % Lastly, Figure~\ref{fig:five-dim-sys}(c) captures the control effort $\|\pi(t,x)\|$ versus time for different number of state measuring instances. We see that when $N=2$, the control effort is significantly high. This is because the controller has no reason to follow the forward process in reverse as we only estimate the initial and final KL divergence. However, when $N \geq 10$, we see that the control effort is both smoother and smaller as the forward process provides a smooth trajectory for the system to track in reverse. 
\subsubsection{Unicycle robot}
In this experiment, we consider the unicycle dynamics,
\begin{align}\label{eqn:unicycle}
\begin{aligned}
\dot{x}_1 &= u_1 \cos(x_3), \ \
\dot{x}_2 = u_1 \sin(x_3), \ \
\dot{x}_3 = u_2.    
\end{aligned}
\end{align}
We sample $M$ training samples from $p_\text{target} = \mathcal{N}(4,0.2I)$ and $p_{\text{initial}} = \mathcal{N}(0,I)$. Note that this experiment demonstrates that our DDPM-based feedback algorithm can be applied for cases with $V(x) \neq 0$. Figure~\ref{fig:unicycle-expts}(a) depicts the performance of the controller for different number of measurement instances. It shows that with fewer measurement instances, it has more difficult for the controller to track the forward density. Figure~\ref{fig:unicycle-expts}(b) further validates our theoretical result even when $V(x)\neq0$. With higher number of training samples the neural network indeed learns to minimize the final KL divergence of the densities. 

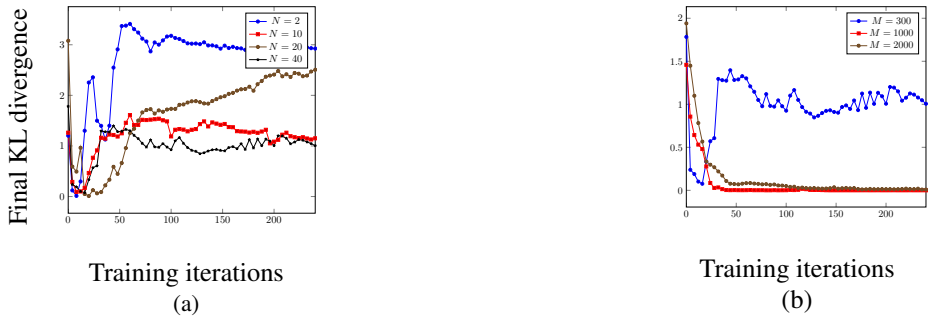
\begin{figure}[tbh]
\centering
\begin{multicols}{2}
\begin{tikzpicture}
  \node (img1)  {\includegraphics[width=0.2100\textwidth]{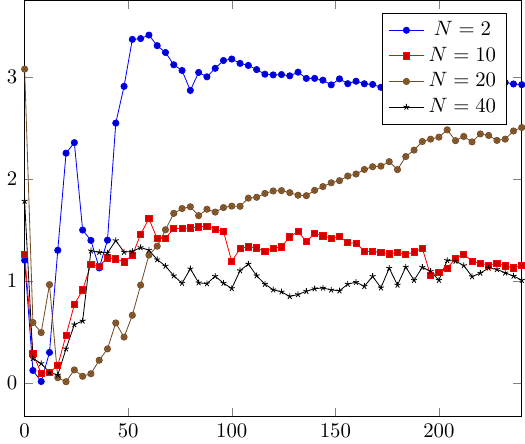}};
  \node[above of= img1, node distance=0cm, yshift=-2.1cm,font=\color{black}]  {Training iterations};  
  % \node[below of= img1, node distance=0cm, yshift=-1.7cm,font=\color{black}]  {\small Accuracy};
  \node[above of= img1, node distance=0cm, yshift=-2.5cm,font=\color{black}]  {\small (a)};
  \node[left of= img1, node distance=0cm, rotate=90, anchor=center,yshift=2.2cm,font=\color{black}] { Final KL divergence };
\end{tikzpicture}\columnbreak
\hspace{2mm}
\begin{tikzpicture}
  \node (img1)  {\includegraphics[width=0.2100\textwidth]{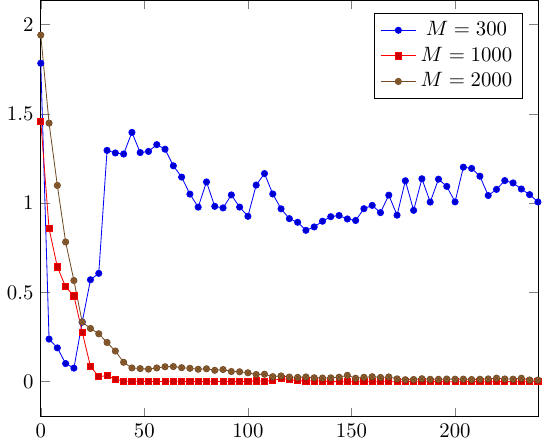}};
  \node[above of= img1, node distance=0cm, yshift=-2.1cm,font=\color{black}]  {Training iterations}; 
  % \node[below of= img1, node distance=0cm, yshift=-1.7cm,font=\color{black}]  {\small $\log$ FPR};
  \node[above of= img1, node distance=0cm, yshift=-2.5cm,font=\color{black}]  {(b)};
  % \node[left of= img1, node distance=0cm, rotate=90, anchor=center,yshift=2.8cm,font=\color{black}] {Cost};
\end{tikzpicture}
\end{multicols}
\vspace*{-2em}
\caption{Experiments with unicycle dynamics~\eqref{eqn:unicycle} with $p_\text{target} = \mathcal{N}(4,0.2I)$. (a) Final estimated KL divergence for different number of measurement instances $N$ vs. training iterations: shows that more measurement instances are required to achieve better feedback control when going from one Gaussian distribution to another, (b) Final KL divergence vs training iterations for different number of training samples: shows that the neural network can identify the controller with sufficiently large number of training samples and state measuring instances.}
\label{fig:unicycle-expts}
\end{figure}

\subsubsection{Husky robots}
We experimentally validate our controller on Husky robots on PyBullet. We spawn four Husky robots whose initial positions are drawn from a uniform distribution with domain $\Omega=(-20,20)^2$. Here, we use the policy learnt from the unicycle dynamics to control the Husky robots individually to reach $p_\text{target} = \mathcal{N}(0,0.2I)$. Figure~\ref{fig:husky-expts} depicts the initial and final positions of the different robots. It can be seen that the robots, in fact, reach the a neighbourhood close to the origin which is denoted by the black pole. It is important to note that the learnt feedback controller performs well even with drift and damping incorporated in the physics-engine.
\begin{figure}[tbh]
\centering
% \begin{multicols}{2}
\begin{tikzpicture}
  \node (img1)  {\includegraphics[width=0.4200\textwidth]{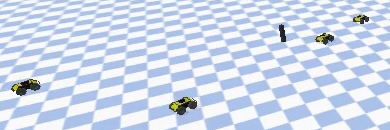}};
\end{tikzpicture}\\
\begin{tikzpicture}
  \node (img1)  {\includegraphics[width=0.4200\textwidth]{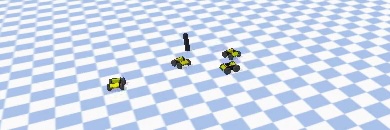}};
\end{tikzpicture}
% \end{multicols}
% \vspace*{-2em}
\caption{Experiments with Husky robots (Top) Husky robots with initial positions sampled from a uniform distribution, (Bottom) Final positions of the different robots sampled from a neighbourhood around the origin.}
\label{fig:husky-expts}
\end{figure}

\section{Conclusion} 
We presented a novel DDPM-based feedback control method for nonlinear control systems. We used the forward noising process of DDPMs as a trajectory to track in reverse for the control system such that the final state is sampled from the target set. We theoretically show that nonlinear control affine systems can indeed track the forward process in reverse subject to a controllability condition on the Lie brackets of the system. We numerically validate our method on various nonlinear systems. There are many open lines for future work. One key direction is to use a forward process that is compatible with the dynamics of the control system. Another direction is to solve the infinite-horizon control problem by combining sequences of finite-time DDPM-based control problem.

\bibliography{refs}
\bibliographystyle{plain}

%%%%%%%%%%%%%%%%%%%%%%%%%%%%%%%%%%%%%%%%%%%%%%%%%%%%%%%%%%%%%%%%%%%%%%%%%%%%%%%
%%%%%%%%%%%%%%%%%%%%%%%%%%%%%%%%%%%%%%%%%%%%%%%%%%%%%%%%%%%%%%%%%%%%%%%%%%%%%%%
% APPENDIX
%%%%%%%%%%%%%%%%%%%%%%%%%%%%%%%%%%%%%%%%%%%%%%%%%%%%%%%%%%%%%%%%%%%%%%%%%%%%%%%
%%%%%%%%%%%%%%%%%%%%%%%%%%%%%%%%%%%%%%%%%%%%%%%%%%%%%%%%%%%%%%%%%%%%%%%%%%%%%%%
\newpage
\appendix
\onecolumn
\section{Analysis}
\label{sec:app}

We define some new notation that will be used in this section. Let $L^{\perp}_2(\Omega) := \{ f \in L^2(\Omega);\int_{\Omega} f(x)dx = 0\}$ be the subspace of functions in $L^2(\Omega)$ that integrate to $0$. We define the {\it horizontal Sobolev space} $WH^1(\Omega) = \big \lbrace f \in L^2(\Omega): Y_if \in L^2(\Omega) \text{ for } 1 \leq i \leq m \big \rbrace$. We equip this space with the horizontal Sobolev norm $\|\cdot\|_{WH^1}$, given by $\|f\|_{WH^1} = \Big( \|f\|^2_{2} + \sum_{i=1}^n\|  Y_if\|^2_{2}\Big)^{1/2} \nonumber$ for each $f \in WH^1(\Omega)$. Here, the derivative action of $X_i$ on a function $f$ is to be understood in the distributional sense. If $C>0$ is a constant, then for $p \in L^2(\Omega)$, $p>C$ will mean that $p(x)>c$ for almost every $x \in \Omega$.

Let $\mathcal{D}(\omega) = WH^1(\Omega)$, we define the sesquilinear form $\omega:\mathcal{D}(\omega) \times \mathcal{D}(\omega) \rightarrow \mathbb{C}$ as
\begin{equation}
	\label{eq:Clpgenform1hol}
	\omega(u,v) = \sum_{i= 1}^m\int_{\Omega}  Y_i  u (\mathbf{x}) \cdot Y_i \bar{v} (\mathbf{x})d\mathbf{x}
\end{equation} 
for each $u,v \in \mathcal{D}(\omega)$. 
We associate with the form $\omega$ an operator $\A :\mathcal{D}(\A)  \rightarrow  L^2(\Omega)$, defined as $ \A u = v$, if $\omega(u,\phi) = \langle v , \phi \rangle $ for all $\phi \in \mathcal{D}(\omega)$ and for all $u \in \mathcal{D}(A) = \lbrace g \in \mathcal{D}(\omega): ~ \exists f \in L^2(\Omega) ~ \text{s.t.} ~  \omega(g,\phi)= \langle f, \phi \rangle_a ~ \forall \phi \in \mathcal{D}(\omega) \rbrace$. Note that the operator $\A$ is a weak formulation of the the second-order partial differential operator $\sum_{i =1}^m Y^*_i  Y_i ( \cdot ~)$. 
% \[\mathbf{n}\cdot (\nabla_H y - \sum_{i=1}^mu_i(\mathbf{x}) X_iy)= 0  & ~~in ~~ \partial  \Omega \times [0,T] \]

When $Y_i = \partial_{x_i}$ are coordinate vector-fields, the operator $\A$ is the negative of the Laplacian operator $-\Delta  = -\sum_{i=1}^d \partial_{x^2_i} $, and for this special case we will denote the operator by $\B$.
Given these definitions, we note the following classical properties of the heat equation.

% \begin{align}

% \end{align}

\label{sec:analysis}

\begin{proposition}[\textbf{Properties of the heat equation}]
\label{prop:hteqprop}
Let $V \equiv 0$. Given Assumption \ref{asmp1}, let $p_0 \in L^2(\Omega)$, then there exists a (mild) solution $p_t \in C([0,T];L^2(\Omega)) $ to the heat equation \eqref{eq:fwdpdf}. Moreover, the solution satisfies the following properties.
\begin{enumerate}
\item There exists a semigroup of operators $(\T(t))_{t \geq 0}$ such that the solution of the heat equation \eqref{eq:fwdpdf} is given by $\T(t)p_0$.
\item $p_t= \T(t) p_0 \in \D(\B)$ for all $ t \in (0,T]$. 
\item If $p_{\text{target}} \in \D(\B)$, then $\dot{p} = -\B p$ for all $t \in [0,T]$.
\item The asymptotic stability estimate holds 
\begin{equation}
\|\T(t)-c\mathbf{1}_\Omega\| \leq Me^{-\omega t} \|p_0 - c\mathbf{1}_\Omega\|
\end{equation}
for all $ t \geq 0$, where $c = 1/\int_{\Omega}dx$, for some $M, \omega>0$ independent of $p_0$.
\end{enumerate}
\end{proposition}
\begin{proof}
The existence of and exponential stability of the solution of the heat equation is classical \cite{ouhabaz2009analysis}. The differentiability of the solution with respect to time is also well known. See for example, \cite{curtain2012introduction}[Theorem 2.1.10]
\end{proof}

In the following proposition we establish the invertibility of the operator $\sum_{i=1}^mY_i^* Y_i$.

\begin{proposition} 
\label{prop:invA}
Given Assumption \ref{asmp2} and \ref{asmp1}, 
	Suppose $f \in L^2(\Omega)$ such that $\int_{\Omega} f(x) dx =0$. Then there exists a unique solution $\phi \in WH^1(\Omega)$ to the nonholonomic Poisson equation,
\begin{equation}
\label{eq:hpoisseq}
\sum_{i=1}^mY_i^* Y_i \phi = f.
\end{equation}
Hence, the restriction of the operator $\A$ on $L^2_{\perp}(\Omega)$, is invertible.
\end{proposition}
\begin{proof}
	Consider the operator $\A = \sum_{i=1}^m Y_i^*Y_i$. It is easy to see that $\mathbf{1}$ is an eigenvector of $\A$ since $Y_i\mathbf{1}=\mathbf{0}$.  We know \cite{elamvazhuthi2023density} from that $\mathbf{1}$ is an eigenvector is an simple eigenvalue. Moreover, since the domain $\Omega$ satisfies Assumption \ref{asmp1} the domain is also NTA \cite{monti2005non}[Theorem 1.1], and hence also $\epsilon-\delta$ in the sense of \cite{garofalo1998lipschitz}. Therefore, from \cite{elamvazhuthi2023density}[Lemma III.3], the spectrum of $\A$ is purely discrete, consisting only of eigenvalues of finite multiplicity and have no finite accumulation point. Let $(\lambda_n)_{n=1}^{\infty}$ be the eigenvalues corresponding to the orthogonal basis of eigenvectors $\{e_n;n \in \mathbb{N}\}$. 
 
Since the operator $\A$ is positive and self-adjoint the eigenvalues are ordered in the form $0=\lambda_1< \lambda_2....$, with $\lim_{ n \rightarrow \infty} \lambda_n= +\infty$.

 Let $u \in L^2_\perp(\Omega) $. Then $u = \sum_{n=2}^\infty \alpha_n e_n$ for some unique sequence $(\alpha_n)_{n=1}^{\infty}$ such that $\sum_{n=2}^{\infty} |\alpha_n|^2 = \|u\|^2_2 $.
From this we can compute that
\begin{align}
\label{eq:coerK}
\langle Au,u \rangle_2 = \sum_{n=2}^{\infty} \lambda_i |\alpha_n|^2 
\geq   \frac{1}{\lambda_2}  \sum_{n=2}^{\infty}  |\alpha_n|^2 =  \frac{1}{\lambda_2}  \|u\|^2_2
\end{align}
Define the space $WH^1_{\perp}(\Omega) : = WH^1(\Omega) \cap L^2_{\perp}(\Omega)$ equiped with the norm $\|u\|_{\omega} : = \sqrt{\omega(u,u)}$.

 The solvability of equation \eqref{eq:hpoisseq} is equivalent to finding a solution $\phi \in WH^1_{\perp}(\Omega) $ such that 
\[\omega(\phi,v) = <f,v>_2 ,~~~\forall v \in \mathcal{D}(\omega)\]
It is easy to see that 
\[\omega(u,v) \leq \|u\|_\omega \|v\|_{\omega}\]
for all $u,v \in  WH^1_{\perp}(\Omega) $
Moreover, from the (Poincare inequality) estimate \eqref{eq:coerK} that 
\[\omega(u,u) \geq \frac{1}{\lambda_2}\|u\|_2\]
for all $u \in  WH^1_{\perp}(\Omega) $. Hence, existence and uniqueness of the solution $\phi \in WH^1_{\perp}(\Omega)$ from the Lax-Milgram theorem \cite{evans2022partial}.
\end{proof}

Using the above result on invertibility of the operator $\mathcal{A} = \sum_{i=1}^m Y_i^* Y_i$, we can establish that that any sufficiently regular trajectory on the set of probability densities can be tracked using the control system in reverse.

 {\it Proof of Lemma} \ref{lem:extrafull}. Let $\phi_t = \A^{-1}\partial_t p_{r}$, for all $ t \in [0,T]$. Then the control law law $u_i = \frac{Y_i \phi_t}{p}$ is well defined by Proposition \ref{prop:invA} and a solution $p^{\textup{c}} = p_{r}$ for all $t \in [0,T]$.\hfill $\qed$
% \square 

\begin{lemma}[\textbf{Tracking the Heat Equation}]
\label{lem:extra}
Given Assumption \ref{asmp1}, 
	suppose $p_0  \in L^2(\Omega)$ is a probability density and $c,C >$ are constants such  $c \leq p_0 \leq C$. Then  $c \leq p_t \leq C$ for all $t \in [0,T]$.

 Suppose additionally, that $p_0 \in \D(\B)$, then there exists a control law $u_i  \in C([0,T]; L^2(\Omega))$ such that a solution $p^{\textup{c}}$ of the \eqref{eq:ctctty} satisfies 
\[p^{\textup{c}}_t = p^{\textup{f}}_{T-t} \text{ for all } t\in [0,T]\] 
\end{lemma}

\begin{proof}
It is a well known property of the solutions of the heat equation that given the assumption of lower and upper bound on the initial condition, the solution $p$ of the heat equation \eqref{eq:fwdpdf}  satisfies $c \leq  p(t,\cdot) \leq C$ for all $t \in [0,T]$. See for example, \cite{elamvazhuthi2018bilinear}[Corollary IV.2 and Lemma IV.5]. 
Moreover, if $p_0 \in \D(\B)$, we know that $ \dot{p} = -\B p $, from Proposition \ref{prop:hteqprop}. This implies that $\langle \mathbf{1}, \dot{p}(t) \rangle_2 = \langle B \mathbf{1}, \dot{p}(t) \rangle_2 = \mathbf{0}$
Hence, $\dot{p}(t) \in L^2_{\perp} (\Omega)$ for all $ t \in [0,T]$. 

\end{proof}

Now we are ready to provide the proof of our main result.

{\it Proof of Theorem \ref{thm:tackdens}}.
First, we consider the case $p_0 \in L^2(\Omega)$. Then we know that $p_t \in \D(\B)$ for all $t \in (0,T]$. It is known that the semigroup generated by the $\B$ is $L^2 - L^{\infty}$ contractive. That is, for every $t >0$, there exists $M_t> 0$ such that $ \|\mathcal{T}(t) p_0 \|_\infty \leq M_t \|p_0\|_2 $ for all $t$. The semigroup $\T(t)$ is {\it irreducible} as an semigroup of operators on $L^2(\Omega)$. That is, $\T(t) p_0 $ is positive almost everywhere on $\Omega$, provided that $p_0$ is non-negative almost everywhere on $\Omega$.    Moreover, $\T(t) \mathbf{1} = \mathbf{1}$. Then it follows from \cite{gluck2019almost}[Proposition 2.21], that $\T(t)$ is also irreducible as an operator on $L^{\infty}(\Omega)$. This implies that there exists, for each $t \in (0,T]$ there exists $c_t>0$ such that $(\T(t) p_0)(x) \geq c_t$ for almost every $x \in \Omega$.
Now, we have that for each $t \in (0,T]$, there exist constants $c_t, C_t >0$ such that $c_t
\leq p_t \leq C_t$. The result for $p_0 \in L^2(\Omega)$ then follows from Lemma \ref{lem:extra}. 

For the case when $p_{\text{target}} = \delta_{x_0}$, the result follows from \cite{santambrogio2015optimal}[Proposition 8.10], there exists a solution to \eqref{eq:fwdpdf} in $AC([0,T];\mathcal{P}_1(\Omega))$, where $AC([0,T];\mathcal{P}_1(\Omega))$ is set of absolutely continuous curves on probability measures with finite moment, and that $p_t \in L^1(\Omega)$ for all $t \in (0,T]$. From \cite{davies1989heat}[Lemma 2.1.2] we know that there exists a heat kernel $K:[0,T] \times \Omega \times \Omega \rightarrow \R$ such that $T_t p_0 = \int_{\Omega}K(t,x,y)p_0(x)dx $  and it satisfies the bounds
\begin{equation}
0 \leq |K(t,x,y)| \leq a_t 
\end{equation}
for each $t \in (0,T]$ for some constant $a_t>0$, depending on $t>0$. This implies that in fact $p_t \in L^2(\Omega)$ for all $t \in (0,T]$ and the result once again, follows from Lemma \ref{lem:extra}.

%%%%%%%%%%%%%%%%%%%%%%%%%%%%%%%%%%%%%%%%%%%%%%%%%%%%%%%%%%%%%%%%%%%%%%%%%%%%%%%
%%%%%%%%%%%%%%%%%%%%%%%%%%%%%%%%%%%%%%%%%%%%%%%%%%%%%%%%%%%%%%%%%%%%%%%%%%%%%%%

\end{document}